\documentclass[10pt]{article}
\usepackage{amsmath,amssymb,amsthm,bbm,verbatim,mathrsfs,graphics,graphicx,epstopdf}

\newcommand{\half}{\mbox{$\frac{1}{2}$}}
\newcommand{\R}{\mathbb{R}}
\newcommand{\DOT}{\pmb{\cdot}}

\newcommand{\cycsum}{\sum_{\pmb{\circlearrowright}}}

\newtheorem{thm}{Theorem}

\newtheorem{lem}{Lemma}

\theoremstyle{definition}

\title{A Simple Vector Proof of Feuerbach's Theorem}
\author{Michael Scheer}
\date{}

\begin{document} 
\maketitle

\begin{abstract}
The celebrated theorem of Feuerbach states that the nine-point circle of a nonequilateral triangle is tangent to both its incircle and its three excircles. In this note, we give a simple proof of Feuerbach's Theorem using straightforward vector computations. All required preliminaries are proven here for the sake of completeness.
\end{abstract}

\section{Notation and Background}

Let $\triangle ABC$ be a nonequilateral triangle. We denote its side-lengths by $a,b,c$, its semiperimeter by $s = \half (a + b + c)$, and its area by $K$. Its {\em classical centers} are the circumcenter $O$, the incenter $I$, the centroid $G$, and the orthocenter $H$ (Figure \ref{HGIO}). The nine-point center $N$ is the midpoint of $OH$ and the center of the nine-point circle, which passes through the side-midpoints $A',B',C'$ and the feet of the three altitudes. The Euler Line Theorem states that $G$ lies on $OH$ with $OG : GH = 1 : 2$. We write $E_a,E_b,E_c$ for the excenters opposite $A,B,C$, respectively; these are points where one internal angle bisector meets two external angle bisectors. Like $I$, the points $E_a,E_b,E_c$ are equidistant from the lines $AB$, $BC$, and $CA$, and thus center three circles each of which is tangent to those lines. These are the excircles, pictured in Figure \ref{excenters}. The {\em classical radii} are the circumradius $R$ ($= |OA| = |OB| = |OC|$), the inradius $r$, and the exradii $r_a,r_b,r_c$. The following area formulas are well known (see, e.g., \cite{C} and \cite{CG}):
\[ K = \frac{abc}{4R}=rs=r_a(s-a)=\sqrt{s(s-a)(s-b)(s-c)}. \]
Feuerbach's Theorem states that {\em the incircle is internally tangent to the nine-point circle, while the excircles are externally tangent to it} \cite{F}. Two of the four points of tangency can be seen in Figure \ref{excenters}.

\section{Vector Formalism}

We view the plane as $\R^2$ with its standard vector space structure. Given $\triangle ABC$, the vectors $A - C$ and $B - C$ are linearly independent. Thus for any point $X$, we may write $X - C = \alpha(A - C) + \beta(B - C)$ for unique $\alpha, \beta \in \R$. Defining $\gamma = 1 - \alpha - \beta$, we find that
\[ X = \alpha A + \beta B + \gamma C, \;\;\;\;\;\; \alpha + \beta + \gamma = 1. \]
This expression for $X$ is unique. One says that $X$ has {\em barycentric coordinates} $(\alpha, \beta, \gamma)$ with respect to $\triangle ABC$ (see, e.g., \cite{C}). The barycentric coordinates are particularly simple when $X$ lies on a side of $\triangle ABC$:

\begin{thm}\label{XonBC} Let $X$ lie on side $BC$ of $\triangle ABC$. Then, with respect to $\triangle ABC$, $X$ has barycentric coordinates $(0,|CX|/a,|BX|/a)$. \end{thm}

\begin{proof} Since $X$ lies on line $BC$ between $B$ and $C$, there is a unique scalar $t$ such that $X - B = t(C - B)$ and $0 < t < 1$. Taking norms and using $t > 0$, we find $|BX| = |t||BC| = ta$, i.e., $t = |BX|/a$. Rearranging, $X = 0A + (1 - t)B + tC$, in which the coefficients sum to $1$. Finally, $1 - t = (a - |BX|)/a = |CX|/a$. \end{proof}

\begin{figure}[h!]
\caption{The classical centers and the Euler division $OG : GH = 1 : 2$.}
\begin{center}
\includegraphics[scale=.375]{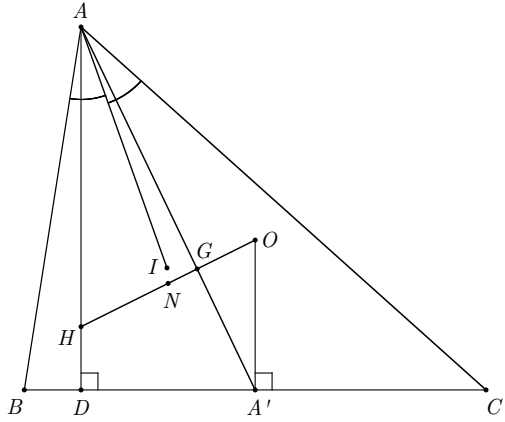}
\end{center}
\label{HGIO}
\end{figure}

\begin{figure}[h]
\caption{The excenter $E_a$ and $A$-excircle; Feuerbach's theorem.}
\begin{center}
\includegraphics[width=60mm]{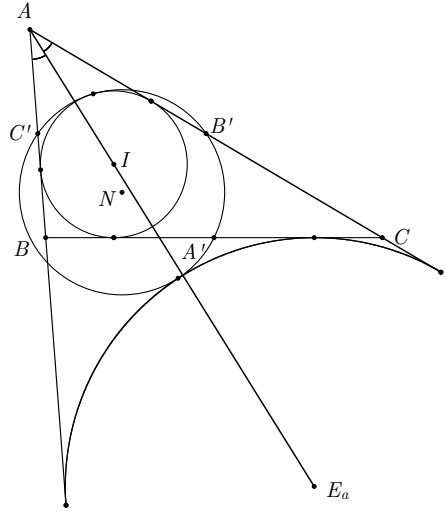}
\end{center}
\label{excenters}
\end{figure}

The next theorem reduces the computation of a distance $|XY|$ to the simpler distances $|AY|$, $|BY|$, and $|CY|$, when $X$ has known barycentric coordinates.

\begin{thm}\label{distance}
Let $X$ have barycentric coordinates $(\alpha, \beta, \gamma)$ with respect to $\triangle ABC$. Then for any point $Y$, 
\[ {|XY|}^2 = \alpha {|AY|}^2 + \beta {|BY|}^2 + \gamma {|CY|}^2 - (\beta \gamma a^2 + \gamma \alpha b^2 + \alpha \beta c^2).  \eqno{(*)} \]
\end{thm}

\begin{proof} Using the common abbreviation $V^2 = V \DOT V$, we compute first that
\begin{align*}
{|XY|}^2 &= {(Y - X)}^2 \\
&= {(Y - \alpha A - \beta B - \gamma C)}^2 \\
&= {\{ \alpha(Y - A) + \beta(Y - B) + \gamma(Y - C) \}}^2 \\
&= \alpha^2 {|AY|}^2 + \beta^2{|BY|}^2 + \gamma^2 {|CY|}^2 \\
& \hspace{.5in} + 2 \alpha \beta \, (Y - A) \DOT (Y - B) + 2 \alpha \gamma \, (Y - A) \DOT (Y - C) \\
& \hspace{.5in} + 2 \beta \gamma \, (Y - B) \DOT (Y - C).
\end{align*}
On the other hand, we may compute $c^2$ as follows:
\[ {(B - A)}^2 = {\{(Y - A) - (Y - B)\}}^2 = {|AY|}^2 + {|BY|}^2 - 2 \, (Y - A) \DOT (Y - B). \]
Thus $2 \alpha \beta \, (Y - A) \DOT (Y - B) = \alpha \beta {|AY|}^2 + \alpha \beta {|BY|}^2 - \alpha \beta c^2$. Substituting this and its analogues into the preceding calculation, the total coefficient of ${|AY|}^2$ becomes $\alpha^2 + \alpha \beta + \alpha \gamma = \alpha(\alpha + \beta + \gamma) = \alpha$, e.g. The result is formula $(*)$.
\end{proof}

\section{Distances from $N$ to the Vertices}

\begin{lem}\label{Gbarys} The centroid $G$ has barycentric coordinates $(\frac{1}{3}, \frac{1}{3}, \frac{1}{3})$. \end{lem}

\begin{proof}
Let $G'$ be the point with barycentric coordinates $(\frac{1}{3}, \frac{1}{3}, \frac{1}{3})$, and we will prove $G=G'$. Let $A'$ and $B'$ be the midpoints of sides $BC$ and $AC$ respectively. By Theorem \ref{XonBC}, $A' = \half B + \half C$. Now we calculate \[ \mbox{$\frac{1}{3}$}A+\mbox{$\frac{2}{3}$}A'  = \mbox{$\frac{1}{3}$}A+\mbox{$\frac{2}{3}$}(\half B + \half C)=\mbox{$\frac{1}{3}$}A + \mbox{$\frac{1}{3}$}B + \mbox{$\frac{1}{3}$}C=G'\]
which implies that $G'$ is on segment $AA'$. Similarly, we find that $G'$ is on segment $BB'$. However the intersection of lines $AA'$ and $BB'$ is $G$, and so $G=G'$.
\end{proof}

\begin{lem}\label{EulerDivision}\emph{(Euler Line Theorem)}
$H-O=3(G-O)$
\end{lem}
\begin{proof}
Let $H'=O+3(G-O)$ and we will prove $H=H'$. By Lemma \ref{Gbarys}, \[H'-O=3(G-O)=A+B+C-3O=(A-O)+(B-O)+(C-O).\] We use this to calculate 
\begin{align*}
(H'-A)\DOT(B-C) &= \{(H'-O)-(A-O)\}\DOT\{(B-O)- (C-O)\} \\
&= \{(B-O)+(C-O)\}\DOT\{(B-O)-(C-O)\} \\
&= |BO|^2-|CO|^2 \\
&= 0
\end{align*}
Therefore $H'$ is on the altitude from $A$ to $BC$. Similarly, $H'$ is on the altitude from $B$ to $AC$, but since $H$ is defined to be the intersection of the altitudes, it follows that $H=H'$.
\end{proof}

\begin{lem}\label{AOdotBO} $(A - O) \DOT (B - O) = R^2 - \half c^2$. \end{lem}

\begin{proof}
One has
\begin{align*} c^2 &= {(A - B)}^2 \\
&= {\{(A - O) - (B - O)\}}^2 \\
&= {|OA|}^2 + {|OB|}^2 - 2 \, (A - O) \DOT (B - O) \\
& = 2R^2 - 2 \, (A - O) \DOT (B - O). \qedhere \end{align*}
\end{proof}

We now calculate $|AN|$, $|BN|$, $|CN|$, which are needed in Theorem \ref{distance}.

\begin{thm}\label{distanceAN} $4{|AN|}^2 = R^2-a^2+b^2+c^2$. \end{thm}

\begin{proof}
Since $N$ is the midpoint of $OH$, we have $H - O = 2(N - O)$. Combining this observation with Lemma \ref{EulerDivision}, and using Lemma \ref{AOdotBO}, we obtain
\begin{align*} 4{|AN|}^2 &= {\{2(A - O) - 2(N - O)\}}^2 \\
&= {\{ (A - O) - (B - O) - (C - O)\}}^2 \\
&= {|AO|}^2 + {|BO|}^2 + {|CO|}^2 \\
& \hspace{.3in} - 2 \, (A - O) \DOT (B - O) - 2 \, (A - O) \DOT (C - O) \\
& \hspace{.3in} + 2 \, (B - O) \DOT (C - O) \\
&= 3R^2 - 2(R^2 - \half c^2) - 2(R^2 - \half b^2) + 2(R^2 - \half a^2) \\
&= R^2 - a^2 + b^2 + c^2. \qedhere  \end{align*}
\end{proof}

\section{Proof of Feuerbach's Theorem}

\begin{thm}\label{Ibarys} The incenter $I$ has barycentric coordinates $(a/2s, b/2s, c/2s)$. \end{thm}

\begin{proof}
Let $I'$ be the point with barycentric coordinates  $(a/2s, b/2s, c/2s)$, and we will prove $I=I'$. Let $F$ be the foot of the bisector of $\angle A$ on side $BC$. Applying the Law of Sines to $\triangle ABF$ and $\triangle ACF$, and using $\sin(\pi - x) = \sin x$, we find that
\[ \frac{|BF|}{c} = \frac{\sin(\angle BAF)}{\sin(\angle BFA)} = \frac{\sin(\angle CAF)}{\sin(\angle CFA)} = \frac{|CF|}{b}. \]
The equations $b|BF| = c|CF|$ and $|BF| + |CF| = a$ jointly imply that $|BF| = ac/(b + c)$. By Theorem \ref{XonBC}, $F = (1 - t)B + tC$, where $t = |BF|/a = c/(b + c)$. Now, \[\mbox{$\frac{b+c}{2s}$}F+\mbox{$\frac{a}{2s}$}A=\mbox{$\frac{b+c}{2s}$}(\mbox{$\frac{b}{b+c}$}B+\mbox{$\frac{c}{b+c}$}C)+\mbox{$\frac{a}{2s}$}A=\mbox{$\frac{a}{2s}$}A+\mbox{$\frac{b}{2s}$}B+\mbox{$\frac{c}{2s}$}C=I'\] which implies that $I'$ is on the angle bisector of $\angle A$. Similarly, $I'$ is on the angle bisector of $\angle B$, but since $I$ is the intersection of these two lines, this implies $I=I'$.
\end{proof}

We are now in a position to prove Feuerbach's Theorem.

\begin{thm}[Feuerbach, 1822]\label{FT} In a nonequilateral triangle, the nine-point circle is internally tangent to the incircle and externally tangent to the three excircles. (For historical details, see {\em \cite{F}} and {\em \cite{M}}.) \end{thm}

\begin{proof} Consider the incircle. From elementary geometry, two nonconcentric circles are internally tangent if and only if the distance between their centers is equal to the absolute difference of their radii. Therefore we must prove that $|IN| = |\half R - r|$. Here, $\half R$ is the radius of the nine-point circle, as the latter is the circumcircle of the midpoint-triangle $\triangle A'B'C'$. We set $X = I$ and $Y = N$ in Theorem \ref{distance}, with Theorems \ref{distanceAN} and \ref{Ibarys} supplying the distances $|AN|$, $|BN|$, $|CN|$, and the barycentric coordinates of $I$. For brevity, we use {\em cyclic sums}, in which the displayed term is transformed under the permutations $(a,b,c)$, $(b,c,a)$, and $(c,a,b)$, and the results are summed (thus, symmetric functions of $a,b,c$ may be factored through the summation sign, and $\sum_{\circlearrowright} a = a + b + c = 2s$). The following computation results:
\begin{align*}
{|IN|}^2 &= \cycsum \bigg(\frac{a}{2s}\bigg) \frac{R^2 - a^2 + b^2 + c^2}{4}  \, - \cycsum  \bigg(\frac{b}{2s} \cdot \frac{c}{2s} \bigg) a^2 \\
&= \frac{R^2}{8s} \bigg[ \cycsum a \bigg] + \frac{1}{8s} \bigg[ \cycsum (-a^3 + ab^2 + ac^2) \bigg] - \frac{abc}{(2s)^2} \bigg[ \cycsum a \bigg] \\
&= \frac{R^2}{4} + \frac{(- a + b + c)(a - b + c)(a + b - c) + 2abc}{8s} - \frac{abc}{2s} \\
&= \frac{R^2}{4} + \frac{(2s - 2a)(2s - 2b)(2s - 2c)}{8s} - \frac{abc}{4s} \\
&= \frac{R^2}{4} + \frac{(K^2/s)}{s} - \frac{4RK}{4s} \\
& = \left(\half R\right)^2 + r^2 - Rr \\
& = \left(\half R - r\right)^2.
\end{align*}
The two penultimate steps use the area formulas of Section 1---in particular, $K = rs = abc/4R$ and $K^2 = s(s - a)(s - b)(s - c)$. A similar calculation applies to the $A$-excircle, with two modifications: (i) $E_a$ has barycentric coordinates
\[ \left( \frac{-a}{2(s - a)}, \frac{b}{2(s - a)}, \frac{c}{2(s - a)} \right), \]
and (ii) in lieu of $K = rs$, one uses $K = r_a(s - a)$. The result, $|E_a N| = \half R + r_a$, means that the nine-point circle and the $A$-excircle are externally tangent.
\end{proof}

\section{Acknowledgements} 

The author wishes to thank his teacher, Mr.\ Joseph Stern, for offering numerous helpful suggestions and comments during the writing of this note.

\bigskip

\noindent {\bf Michael Scheer} is a secondary student now entering his senior year at Stuyvesant High School, the technical academy in New York City's TriBeCa district. \\
\texttt{mscheer@stuy.edu}
\end{document}